\documentclass{birkjour}
 \newtheorem{thm}{Theorem}[section]

 \newtheorem{prop}[thm]{Proposition}
 \theoremstyle{definition}
 
 \theoremstyle{remark}
 \newtheorem{rem}[thm]{Remark}
 
 \numberwithin{equation}{section}

\begin{document}

%
%
%
%
%
%
%
%
%

\title[Scattering Matrices for Close Singular Selfadjoint Perturbations]
 {Scattering Matrices for Close Singular Selfadjoint Perturbations of Unbounded Selfadjoint Operators}
\author[Adamyan]{Vadym Adamyan }

\address{%
 Department  of Theoretical Physics and Astronomy\\
Odessa National I.I. Mechnikov  University\\
65082 Odessa\\ 
Ukraine}

\email{vadamyan@onu.edu.ua}

\thanks{This paper is a part of the ongoing research project 0120U104119 of the Department of Theoretical Physics and Astronomy of Odessa I.I. Mechnikov National University.} 



\date{}
\dedicatory{}

\begin{abstract}
In this paper, we consider an unbounded selfadjoint operator $A$ and its selfadjoint perturbations in the same Hilbert space $\mathcal{H}$.  As in \cite{AK}, we call a selfadjoint operator $A_{1}$ the singular perturbation of $A$ if $A_{1}$ and {A} have different domains $\mathcal{D}(A),\mathcal{D}(A_{1})$ but $A=A_{1}$ on $\mathcal{D}(A)\cap\mathcal{D}(A_{1})$. Assuming that $A$ has absolutely continuous spectrum and the difference of resolvents $R_{z}(A_{1}) -R_{z}(A)$ of  $A_{1}$ and $A$ for non-real $z$ is a trace class operator we find the explicit expression for the scattering matrix for the pair $A, A_{1}$  through the constituent elements of the Krein formula for the resolvents of this pair. As an illustration, we find the scattering matrix for the standardly defined Laplace operator in $L_{2}\left(\mathbf{R}_{3}\right)$ and its singular perturbation in the form of an infinite sum of zero-range potentials. 
\end{abstract}

\maketitle
\section{Introduction}
Let $A,A_{1}$ be unbounded selfadjoint operators in Hilbert space $\mathcal{H}$, $\mathcal{D},\mathcal{D}_{1}$ and $R(z),R_{1}(z), \, \mathrm{Im}z\neq 0,$ are the domains and resolvents of $A,A_{1}$, respectively. Following \cite{AK} we call $A_{1}$ a \textit{singular perturbation} of $A$ if
 $$ \begin{array}{cc} 
1)\,  \mathcal{D}\cap\mathcal{D}_{1} \, \text{is dense in }  \mathcal{H}; &
2)\, A_{1}=A \, \text{on } \mathcal{D} \cap \mathcal{D}_{1}. \end{array} 	
$$ 
In accordance with this definition $A,A_{1}$ are selfadjoint extensions of the same densely defined symmetric operator
$$ B= A|_{\mathcal{D}\cap\mathcal{D}_{1}}=A_{1}|_{\mathcal{D}\cap\mathcal{D}_{1}} . $$
Therefore, the specific properties of the operator $A_{1}$ as a perturbation of the given operator $A$ can be investigated in the framework of the extension theory  of symmetric operators. However, such a way can be long and involved, since in this case, it is not possible to directly operate with the difference of the singular perturbation $A_{1}$ of $A$ and $A$. Starting from the known properties of the operator $A$, the spectral analysis of the perturbation $A_{1}$ and the scattering theory for the pair $A, A_{1}$ can be developed bypassing the constructions of extension theory by operating directly and exclusively with the resolvent of  $A$ and the constituent elements of  Krein's formula \cite{Kr} for the difference of resolvents $R(z)-R_{1}(z)$. One can do this for a fairly wide class of unbounded selfadjoint operators of mathematical physics and their singular perturbations based on the following version of Krein's formula \cite{Ad}, which combines all the characteristic properties of the resolvents of such perturbations. 
\begin{thm}\label{kreinnext1}
	Let $\mathcal{H}$ and $\mathcal{K}$ be Hilbert spaces, $A$ be an unbounded selfadjoint operator in $\mathcal{H}$ and  $R(z),\; \mathrm{Im}z\neq 0,$ is the resolvent of $A$, $G(z)$  is a bounded holomorphic in the open upper and lower half-planes operator function from $\mathcal{K}$ to  $\mathcal{H}$ satisfying the conditions 
	\begin{itemize}
		\item{for any non-real $z,z_{0}$
			\begin{equation}\label{res3}
			G(z)=G(z_{0})+(z-z_{0})R(z)G(z_{0}), 
			\end{equation}
		}
		\item{ zero is not an eigenvalue of the operator  $G(z)^{*}G(z)$ at least for one and hence for all non-real $z$ and the intersection of the domain $\mathcal{D}(A)$ of $A$ and the subspace $\mathcal{N}=\overline{G(z_{0})\mathcal{K}}\subset\mathcal{H}$ consists only of the zero-vector.}
	\end{itemize}
	Let $Q(z)$ be a holomorphic in the open upper and lower half-planes operator function in $\mathcal{K}$ such that 
	\begin{itemize}
		\item{$Q(z)^{*}=Q(\bar{z}), \; \mathrm{z}\neq 0$;}
		\item{for any non-real $z,z_{0}$
			\begin{equation}
			Q(z)-Q(z_{0})=(z-z_{0})G(\bar{z_{0}})^{*}G(z). 
			\end{equation}
		}
	\end{itemize} 
	Then for any selfadjoint operator $L$ in $\mathcal{K}$ such that the operator $Q(z)+L, \:\mathrm{Im}z\neq 0, $ is boundedly invertible the operator function
	\begin{equation}\label{krein5}
	R_{L}(z)=R(z)-G(z)\left[Q(z)+L \right]^{-1} G(\bar{z})^{*}
	\end{equation}
	is the resolvent of some singular selfadjoint perturbation  $A_{1}$ of $A$. 
\end{thm}

A singular perturbation $A_{1}$ of a given selfadjoint operator $A$ is  hereinafter referred to as \textit{close} (to {A}) if the difference of the resolvents $R_{1}(z)-R(z), \, \mathrm{Im}z\neq 0,$ is a trace class operator. In accordance with this definition, the absolutely continuous components of the close perturbation  and the operator $A$ are similar and absolutely continuous components of $A_{1}$ can be described using the constructions of scattering theory \cite{Kat}. 

The aim of this study is \begin{itemize}
	\item {to derive  an explicit expression relating the scattering matrix for an unbounded selfadjoint operator $A$ with absolutely continuous spectrum and its close singular perturbation $A_{1}$ with operator $L$ and the function $Q(z)$ in the Krein's formula for the same pair $A,A_{1}$, slightly modifying the approach developed in \cite{AdP};} 
	\item {to illustrate the corresponding results by the example of singular selfadjoint perturbations of the Laplace operator $$A=-\Delta=-\frac{\partial^{2}}{\partial x^{2}_{1}}-\frac{\partial^{2}}{\partial x^{2}_{2}}-\frac{\partial^{2}}{\partial x^{2}_{3}} $$  in $\mathbf{L}_{2}(\mathbf{R}_{3})$  defined on the Sobolev subspaces ${H}_{2}^{2}\left( \mathbf{R}_{3}\right)$.}
\end{itemize}

In Section 2, for completeness, we give a sketch of the proof of Theorem \ref{kreinnext1} and indicate  additional conditions under which the difference of  resolvents $R_{1}(z)-R(z)$ of a singular perturbation $A_{1}$ of an unbounded self-adjoint operator $A$ and $A$ itself is a trace class operator, that is the conditions under which the singular perturbation  $A_{1}$ and the unperturbed operator $A$ are close. Further, these conditions are refined for singular perturbations of the Laplace operator in in $\mathbf{L}_{2}(\mathbf{R}_{3})$.

In Section 3, which is some modified version of the paper \cite{AdP}, we calculate the scattering operator for a self-adjoint operator $A$ with an absolutely continuous spectrum and its close singular perturbation $A_{1}$ and find  an explicit expression of the scattering matrix for the pair $A_{1},A$ through the entries in the Krein formula (\ref{krein5}). 
 
 Section 4 is devoted to the scattering theory for self-adjoint singular perturbations of the standardly defined Laplace operator $A$ in $\mathbf{L}_{2}(\mathbf{R}_{3})$ formally given as infinite sum of zero-range potentials. If  such a perturbation $A_{1}$ and $A$ are close, then using the results of previous Section we find an explicit expression for the scattering matrix for this pair. 

  \section{Resolvents of singular selfadjoint perturbations of Laplace operator}
To make the paper self-contained we outline the proof of Theorem \ref{kreinnext1}. 

Recall the folowing well-known criterium (see, for examle \cite{Ad}):\textit{ a holomorphic function $R(z)$ on the open upper and lower half-plains of the complex plane whose values are bounded linear operators in Hilbert space $\mathcal{H}$ is the resolvent of a selfadjoint operator $A$ in $\mathcal{H}$ if and only if 
 \begin{itemize}
	\item{\begin{equation}\label{first1}
		\ker R(z)=\{0\};
		\end{equation}}
	\item{\begin{equation}\label{first}
		R(z)^{*}=R(\overline{z});
		\end{equation}}
	\item{for any non-real $z_{1},z_{2}$ the Hilbert identity 
		\begin{equation}\label{hilb1}
		R(z_{1})-R(z_{2})=(z_{1}-z_{2})R(z_{1})R(z_{2})
		\end{equation} holds.}
	 \end{itemize}
 }

So to verify the validity of Theorem 1.1, it is only necessary to check up that its assumptions about $R_{L}(z)$ guarantee the fulfillment of the conditions (\ref{first1}) - (\ref{hilb1}). 
But relation (\ref{first}) and identity (\ref{hilb1}) for $R_{L}(z)$ directly follow from the fact that $R(z)$ is the resolvent of the self-adjoint operator and the properties that the operator functions $G(z)$ and $Q(z)$ possesses according to the assumptions of Theorem \ref{kreinnext1}. 

To see that (\ref{first1}) is true for $R_{L}(z)$ suppose that there is a vector $h\in \mathcal{H}$  such that $R_{L}(z)h=0$ for some non-real $z$. By (\ref{krein5}) this means that
\begin{equation}\label{krein5a}
R(z)h=G(z)\left[Q(z)+L \right]^{-1} G(\bar{z})^{*}h.
\end{equation}
But by assumption, $R(z)\mathcal{H}\cap G(z)\mathcal{K}=\{0\}$ for non-real $z$. Hence $R(z)h=0$. Given that $R(z)$ is the resolvent of a self-adjoint operator, we conclude from this that $h=0$. Thus, the statement of Theorem \ref{kreinnext1} is true. 

Turning to the description of singular perturbations of the Laplace operator $A=-\Delta$ in $\mathbf{L}_{2}(\mathbf{R}_{3})$ remind that the resolvent $R(z)$ of is the integral operator
\begin{equation}\label{lapl} \begin{array}{c}
\left(R(z)f\right)(\mathbf{x})=\frac{1}{4\pi}
\int_{\mathbf{R}_{3}}\frac{e^{i\sqrt{z}|\mathbf{x}-\mathbf{x}^\prime|}}{|\mathbf{x}-\mathbf{x}^\prime|}f\left(\mathbf{x}^\prime \right)d\mathbf{x}^\prime , \quad \mathrm{Im}{\sqrt{z}}>0, \\ \mathbf{x}=\left(x_{1},x_{2},x_{3} \right) , 
f(\cdot)\in {\mathbf{L}_{2}(\mathbf{R}_{3})}.
\end{array}
\end{equation} 
In the simplest case, when the support of the singular perturbation of the Laplace operator is a finite set of points $\mathbf{x}_{1},...,\mathbf{x}_{N}$ of $\mathbf{R}_{3}$  the $N$-dimensional space $\mathbf{C}_{N}$ may be taken as an auxiliary Hilbert space $\mathcal{K}$ in Krein's formula (\ref{krein5}) for the perturbed resolvent $R_{L}(z)$.  In this case, the linear mappings $G(z)$ of $\mathbf{C}_{N}$ into $\mathbf{L}_{2}(\mathbf{R}_{3}$ , which transforms vectors of the canonical orthonormal basis in $\mathbf{C}_{N}$: $$
\mathbf{e}_{1}=\left( \begin{array}{c} 1 \\ 0 \\ ...\\ 0 \end{array}\right), \dots ,\mathbf{e}_{N}=\left( \begin{array}{c} 0 \\  ... \\0 \\1 \end{array}\right),
$$ respectively, into the functions
\begin{equation}\label{simp1}
\left( G(z)\mathbf{e}_{n}\right)(\mathbf{x})=g_{n}(z;\mathbf{x})= R(z)\delta(\cdot -\mathbf{x}_{n})(\mathbf{x})=\frac{1}{4\pi}\frac{e^{i\sqrt{z}|\mathbf{x}-\mathbf{x}_{n}|}}{|\mathbf{x}-\mathbf{x}_{n}|}, \quad 1\leq n\leq N,
\end{equation}
can be substituted as $G(z)$ into (\ref{krein5}), while, the matrix function 
\begin{equation}\label{simp2}
Q(z)=\left(q_{mn}(z) \right)_{m,n=1}^{N}=\left\lbrace \begin{array}{c} 
q_{mn}(z)=g_{n}(z;\mathbf{x}_{m}-\mathbf{x}_{n}), \quad m\neq n,\\
q_{mm}(z)=\frac{i\sqrt{z}}{4\pi}
\end{array}\right. . 
\end{equation}
may acts as $Q$-function in (\ref{krein5}). As a result, for any invertible Hermitian matrix $L=\left( w_{mn} \right)_{m,n=1}^{N} $ the arising operator function $R_{L}(z)$ satisfies all the conditions of Theorem \ref{kreinnext1} and hence appears to be the resolvent of a singular selfadjoint perturbation of the Laplace operator $A_{L}$ . 

Setting  $\mathcal{N}=G(z_{0})\mathbf{C}_{N},\; \mathrm{Im}z_{0}\neq 0,$ one can easily deduce from the Krein formula for $R_{L}(z)$ that in this case $A_{L}$ is nothing else than the Laplace differential operator $-\Delta$ with the domain \cite{Ad} 
\begin{equation}\label{simp3a}
\begin{array}{c}      
\mathcal{D}_{L}:=  \left\lbrace f: \, f=f_{0}+g, \, f_{0}\in {H}_{2}^{2}\left(\mathbf{R}_{3}\right), g\in\mathcal{N}, \right. \\  
\underset{\rho_{m}\rightarrow 0}{\lim}\left[\frac {\partial} {\partial\rho_{m}}\left(\rho_{m}f(\mathbf{x})\right)\right]+\sum\limits_{n=1}^{N}4\pi w_{mn} 
\underset {\rho_{n}\rightarrow 0}{\lim}\,[\rho_{n}\,f(\mathbf{x})] =0, \\ \left. \rho_{n}=|\mathbf{x}-\mathbf{x}_{n}|, \quad 1\leq n\leq N  \right\rbrace .  \end{array} 
\end{equation}
 \begin{rem}\label{deltacomb}
The boundary conditions (\ref{simp3a}) correspond to singular perturbation $\hat{V}_{L}$, which is formally defined as an operator acting on any continuous function  $f(\cdot)\in \mathbf{L}_{2}(\mathbf{R}_{3})$ by the formula
\begin{equation}\label{deltas}
\left(\hat{V}_{L}f\right)(\mathbf{x})=\sum\limits_{m,n=1}^{N}v_{mn}\delta\left( \mathbf{x}-\mathbf{x}_{m}\right)  
f(\mathbf{x}_{n}), \, \left( v_{mn}\right)_{m,n=1}^{N}=L^{-1}.
\end{equation}	            
If the matrix $L$ is diagonal, then by (\ref{deltas}) we are dealing with a perturbation in the form of a finite sum of zero-range potentials \cite{BF, DeOst} and the corresponding boundary conditions (\ref{simp3a}) legalize the formal expression $$ A_{L}=-\Delta+\sum\limits_{n=1}^{N}w_{n}^{-1}\delta\left( \mathbf{x}-\mathbf{x}_{n}\right).$$
 \end{rem} 
 
 In what follows, it is assumed that the elements of canonical basis $\left\lbrace   \mathbf{e}_{n}=\left(\delta_{nm} \right)_{m\in \mathbb{Z}}\right\rbrace $ of the Hilbert space $\mathbf{l}_{2}$ belong to the domains of all mentioned operators in this space, and no distinction is made between those operators and the matrices that represent them in the canonical basis of $\mathbf{l}_{2}$. 

 For perturbations of the Laplace operator A in the form of an infinite sum of zero-range potentials located at points $\{\mathbf{x}_{n}\}_{m,n\in \mathbb{Z}}$  the set of boundary conditions of the form (\ref{simp3})  also generates a singular self-adjoint perturbation $A_{L}$ of $A$ if
 \begin{equation}\label{dist}
 	\underset{_{m,n\in \mathbb{Z}}}{\inf}\left|\mathbf{x}_{m}-\mathbf{x}_{n} \right|=d>0.   
 \end{equation}

    \begin{thm} [\cite{GHM}]\label{kreinmain1+}
  Let $Q(z)$ be the operator function in $\mathbf{l}_{2}$ generated by the infinite matrix 
 \begin{equation}\label{simp3}
 Q(z)=\left(q_{mn}(z) \right)_{m,n\in \mathbb{Z}}, \, q_{mn}(z)=\left\lbrace \begin{array}{c} 
 q_{mn}(z)=g(z;\mathbf{x}_{m}-\mathbf{x}_{n}), \quad m\neq n,\\
 q_{mm}(z)=\frac{i\sqrt{z}}{4\pi}
 \end{array}\right. . 
 \end{equation}
 and $L$ be a selfadjoint operator in $\mathbf{l}_{2}$  defined by an infinite Hermitian matrix $\left( w_{mn}\right)_{m,n\in \mathbb{Z}}$. If the condition (\ref{dist}) holds, then the operator function
 \begin{equation}\label{krein6a} \begin{array}{c}
  R_{L}(z)=R(z)-\sum\limits_{m,n\in \mathbb{Z}}^{}\left([ Q(z)+4\pi L]^{-1}\right)_{mn} \left(\cdot,g_{n}(\bar{z};\cdot) \right)g_{m}(z;\cdot), \\ g_{n}(z;\mathbf{x})= R(z)\delta(\cdot -\mathbf{x}_{n})(\mathbf{x}), \end{array}
 \end{equation} 
 is the resolvent of the selfadjoint operator $A_{L}$ in $\mathbf{L}_{2}(\mathbf{R}_{3})$, which is the Laplace operator with the domain
 \[
 \mathcal{D}_{L}:=  \left\lbrace f: \, f=f_{0}+g, \, f_{0}\in {H}_{2}^{2}\left(\mathbf{R}_{3}\right), \, g\in \mathcal{N}, \right.  \] 
 
 \[ \left. \begin{array}{c} \underset{\rho_{m}\rightarrow 0}{\lim}\left[\frac {\partial} {\partial\rho_{m}}\left(\rho_{m}f(\mathbf{x})\right)\right]+\sum\limits_{n\in \mathbb{Z}}^{}w_{mn} \underset
 {\rho_{n}\rightarrow 0}{\lim}\,[\rho_{n}\,f(\mathbf{x})] =0, \\ \rho_{n}=|\mathbf{x}-\mathbf{x}_{n}|, \quad n\in \mathbb{Z} .  \end{array}\right.  \]
 \end{thm}
 The above results concerning perturbations of the Laplace operator in the form of a finite or infinite sum of zero-range potentials do not contain anything new compared to those that were previously collected in the books\cite {AGHH, AK}. However, they can be obtained without the use of concepts and labor-consuming constructions of the theory of selfadjoint extensions of symmetric operators (see,\cite{Ad}).

Of course, the sparseness (\ref{dist}) of the set $\{\mathbf{x}_{n}\}_{n\in \mathbb{Z}} $ is not necessary for the expression (\ref{krein6a}) in Theorem \ref{kreinmain1+} to be the resolvent of a self-adjoint perturbation of the Laplace operator in the form of an infinite sum of zero-range potentials. If (\ref{dist}) is not satisfied, then the operator functions $R_{L} (z)$, formally defined as in Theorem \ref{kreinmain1+}, can nevertheless be resolvents of close singular selfadjoint perturbations $A_{L}$ of $A$ for account of special choice of the parameter $L$. 
\begin{thm}\label{help1}
Let $\{\mathbf{x}_{n}\}_{n\in \mathbb{Z}}$ be a sequence of different points of $\mathbf{R}_{3}$ such that each compact domain of $\mathbf{R}_{3}$ contains a finite number of accumulation points of this set. Set
\begin{equation}\label{help2a}
\delta_{0}=|\mathbf{x}_{0}|; \; \delta_{n}=\underset{-n\leq j\neq k \leq n}{\min}|\mathbf{x}_{j}-\mathbf{x}_{k}|,  \; n=1,2,... 
\end{equation} 
Let $L$ be a selfadjoint operator in $\mathbf{l}_{2}$ such that the point $z=0$ is in the resolvent set of $L$ and the matrix $\left( b_{mn}\right)_{m,n\in \mathbb{Z}}$ of operator $|L|^{-\frac{1}{2}}$ in the canonical basis $\left\lbrace \mathbf{e}_{n}=\left(\delta_{nm} \right)_{m\in \mathbb{Z}} \right\rbrace_{n\in \mathbb{Z}} $ of $\mathbf{l}_{2}$ satisfies the conditions:
\begin{equation}\label{cond1}
\sum_{nm\in \mathbb{Z}}\left|b_{nm}\right|<\infty, \quad  \sum_{nm\in \mathbb{Z}}\left|b_{nm}\right|\frac{1}{\delta_{m}}<\infty.
\end{equation}
Then the operator function
\begin{equation}\label{krein7} \begin{array}{c}
		R_{L}(z)=R(z)-\sum_{m,n\in \mathbb{Z}}\left([ Q(z)+4\pi L]^{-1}\right)_{mn} \left(\cdot,g_{n}(\bar{z};\cdot) \right)g_{m}(z;\cdot) \\ =
	R(z)-\tilde{G}(z)\left([ \tilde{Q}(z)+4\pi J_{L}]^{-1}\right)\tilde{G}(\bar{z})^{*},\\ 	
		g_{n}(z;\mathbf{x})= R(z)\delta(\cdot -\mathbf{x}_{n})(\mathbf{x}), \\ \tilde{G}(z)\mathbf{h}=\sum_{n}\left(|L|^{-\frac{1}{2}}\mathbf{h},\mathbf{e}_{n}\right)g_{n}(z;\cdot), \, \mathbf{h}\in \mathbf{l}_{2},\\
		\tilde{Q}(z)=|L|^{-\frac{1}{2}}Q(z)|L|^{-\frac{1}{2}}, \quad J_{L}=L\cdot |L|^{-1},
	\end{array}
\end{equation}
is the resolvent of the close singular perturbation $A_{L}$ of the Laplace operator A 
with the domain
\begin{equation}\begin{array}{c}
\mathcal{D}_{L}:=  \left\lbrace f: \, f=f_{0}+g, \, f_{0}\in {H}_{2}^{2}\left(\mathbf{R}_{3}\right),\right\rbrace \, g\in \mathcal{N},  \\

  \underset{\rho_{m}\rightarrow 0}{\lim}\left[\frac {\partial} {\partial\rho_{m}}\left(\rho_{m}f(\mathbf{x})\right)\right]+4\pi\sum\limits_{n\in \mathbb{Z}}^{}\left( L\mathbf{e}_{n},\mathbf{e}_{m}\right)  \underset
	{\rho_{n}\rightarrow 0}{\lim}\,[\rho_{n}\,f(\mathbf{x})] =0, \\ \rho_{n}=|\mathbf{x}-\mathbf{x}_{n}|, \quad n\in \mathbb{Z} .  \end{array} \end{equation}
\end{thm}
\begin{proof}

First of all, note that the mapping $\tilde{G}(z)$ from $\mathbf{l}_{2}$  to ${L}_{2}\left(\mathbf{R}_{3}\right)$ belongs to the Hilbert-Schmidt class. Indeed, for the canonical basis $\left\lbrace \mathbf{e}_{n} \right\rbrace_{n\in \mathbb{Z}} $ in $\mathbf{l}_{2}$ by virtue of (\ref{cond1}) we have
\begin{equation}\begin{array}{c}
		\sum_{n\in \mathbb{Z}}\left\|\tilde{G}(z)\mathbf{e}_{n}\right\|^{2}=\sum_{n\in \mathbb{Z}}\left(|L|^{-\frac{1}{2}}G(z)^{*}G(z) |L|^{-\frac{1}{2}}\mathbf{e}_{n},\mathbf{e}_{n}\right) \\=\sum_{n,m,m^\prime\in \mathbb{Z}}b_{nm^\prime}\cdot\frac{1}{8\pi}e^{-\mathrm{Im}\sqrt{z}|\mathbf{x}_{m}-\mathbf{x}_{m^\prime}|}\frac{\sin{\mathrm{Re}\sqrt{z}|\mathbf{x}_{m}-\mathbf{x}_{m^\prime}|}}{\mathrm{Re}\sqrt{z}|\mathbf{x}_{m}-\mathbf{x}_{m^\prime}|}\cdot b_{mn} \\
		\leq \frac{1}{8\pi}\sum_{n,m,m^\prime\in \mathbb{Z}}|b_{nm}|\cdot |b_{m^\prime n}| \leq \frac{1}{8\pi}\left( \sum_{n,m\in \mathbb{Z}}|b_{mn}|\right)^{2}<\infty. 
	\end{array}
\end{equation}

Turning to the operator function $\tilde{Q}(z)$, we represent the corresponding matrix function as the sum $D(z)+M(z)$ of matrices
\begin{equation}\label{diag1}
	D(z)=\frac{i\sqrt{z}}{4\pi}\left( \sum_{j\in \mathbb{Z}}b_{mj}b_{jn}\right)_{m,n\in\mathbb{Z}}   
\end{equation} 
and 
\begin{equation}\label{rest1}
		M(z)=\left(u_{mn}(z)=\sum_{n^\prime,m^\prime \in \mathbb{Z}}b_{n n^\prime}g(z;\mathbf{x}_{n^\prime}-\mathbf{x}_{m^\prime})b_{m^\prime m}, \, m^\prime\neq n^\prime \right)_{m,n\in \mathbb{Z}} .
\end{equation}
  	
  	Note that, according to the first assumption in (\ref{cond1}), $|L|^{-\frac{1}{2}} $ is a Hilbert-Schmidt operator, since $$\sum_{n,j\in \mathbb{Z}}b_{nj}b_{jn}=\sum_{n,j\in \mathbb{Z}}|b_{nj}|^{2} \leq \left( \sum_{n,j\in \mathbb{Z}}|b_{nj}|\right) ^{2}<\infty.$$
As follows, $D(z)=\frac{i\sqrt{z}}{4\pi}|L|^{-1}$ is a trace class operator.

Taking into account, further, that
\begin{equation}\label{imp1}
\begin{array}{c}
\left|\sum\limits_{m^\prime\neq n^\prime }g(z;\mathbf{x}_{n^\prime}-\mathbf{x}_{m^\prime})b_{m^\prime m}, \,   \right|\leq\sum\limits_{m^\prime\neq n^\prime }\frac{1}{4\pi\left| \mathbf{x}_{n^\prime}-\mathbf{x}_{m^\prime}\right|}|b_{m^\prime m}| \\ 
\leq \frac{1}{4\pi \delta_{n^\prime}}\sum\limits_{|m^{\prime}|\leq|n^{\prime}|}|b_{m^\prime m}|+\sum\limits_{|m^{\prime}|>|n^{\prime}|}\frac{1}{4\pi \delta_{m^\prime}}|b_{m^\prime m}|\leq \frac{1}{\delta_{n^\prime}}C^{(0)}(m)+C^{(1)}(m)_, \\
C^{(0)}(m)=\frac{1}{4\pi}\sum\limits_{j\in \mathbb{Z}}\left| b_{mj}\right|=\frac{1}{4\pi}\sum\limits_{j\in \mathbb{Z}}\left|b_{jm}\right|, \\ C^{(1)}(m)=\frac{1}{4\pi}\sum\limits_{j\in \mathbb{Z}}\left| b_{mj}\right|\frac{1}{\delta_j }=\frac{1}{4\pi}\sum\limits_{j\in \mathbb{Z}}\left|b_{mj}\right|\frac{1}{\delta_j },
\end{array}
\end{equation} 	
and that by (\ref{cond1})
\begin{equation}\label{imp2}
	\sum_{m\in \mathbb{Z}}C^{(0)}(m)<\infty,    \quad \sum_{m\in \mathbb{Z}}C^{(1)}(m)<\infty, 
\end{equation}
we conclude that
\begin{equation}\label{imp3}
	\left|u_{nm} (z)\right|\leq 4\pi\left[C^{(1)}(n)C^{(0)}(m)+C^{(1)}(m)C^{(0)}(n) \right] .  
\end{equation}
Therefore
\begin{equation}
\sum_{nm\in \mathbb{Z}}	\left|u_{nm} (z)\right|^{2}\leq\left( \sum_{nm\in \mathbb{Z}}	\left|u_{nm} (z)\right|\right)^{2}<\infty.
\end{equation}
We see that $M(z)$ is a Hilbert-Schmidt operator and hence the sum $\tilde{Q}(z)=D(z)+M(z)$ is at least a Hilbert-Schmidt operator.

Remember now that $\tilde{Q}(z)+J_{L} $ is an operator function of the Nevanlinna class and that for any $\mathbf{h}\in \mathbf{l}_{2}$ and any non-real $z$ we have
\begin{equation}\label{imp4}\begin{array}{c}
\frac{1}{\mathrm{Im}z}\cdot \mathrm{Im}\left[\left(\tilde{Q}(z)\mathbf{h},\mathbf{h}  \right)+\left( J_{L}\mathbf{h},\mathbf{h} \right) \right] \\ =\frac{1}{\left( 2\pi\right)^{3} }\int_{\mathbf{R}_{3}} \frac{1}{\left(|\mathbf{k}|^{2}-\mathrm{Re}z \right)^{2} +\left( \mathrm{Im}z\right)^{2} }\left|\sum_{n\in \mathbb{Z}}\left(|L|^{-\frac{1}{2}}\mathbf{h},\mathbf{e}_{n} \right)e^{-\mathbf{k}\cdot\mathbf{x}_{n}}\right|^{2} d\mathbf{x}\geq 0
\end{array}		
\end{equation}	

The function 
\begin{equation}\label{imp5}
	\hat{\mathbf{h}}(\mathbf{k})=\sum_{n\in \mathbb{Z}}\left(|L|^{-\frac{1}{2}}\mathbf{h},\mathbf{e}_{n} \right)e^{-\mathbf{k}\cdot\mathbf{x}_{n}}, \; \mathbf{h}\in \mathbf{l}_{2},
\end{equation}
that appears in this case in (\ref{imp4}) is bounded and continuous, since
\begin{equation*}
\sum_{n\in \mathbb{Z}}\left| \left(|L|^{-\frac{1}{2}}\mathbf{h},\mathbf{e}_{n} \right)\right|\leq \sum_{n\in \mathbb{Z}}\left\||L|^{-\frac{1}{2}}\mathbf{e}_{n} \right\|\cdot \left\| \mathbf{h}\right\|\leq \left( \sum_{m,n\in \mathbb{Z}}\left|b_{mn} \right|^{2} \right)^{\frac{1}{2}}\cdot \left\| \mathbf{h}\right\| <\infty.	
\end{equation*} 
Obviously, the equality sign in (\ref{imp4}) is possible if and only if $\hat{\mathbf{h}}(\mathbf{k}) \equiv 0$. But for an expanding family of domains $$ \mathbf{\Omega}_{K}=\left\lbrace\mathbf{k}:-K<k_{1},k_{2},k_{3}<K \right\rbrace \subset \mathbf{R}_{3} $$ and any $n\in \mathbb{Z}$ in this case it turns out that
\begin{equation*}
\underset{K\rightarrow \infty}{\lim}\frac{1}{8K^{3}}\int_{\mathbf{\Omega}_{K}}e^{i\mathbf{k}\cdot\mathbf{x}_{n}}\hat{\mathbf{h}}(\mathbf{k})d\mathbf{k}=\left(|L|^{-\frac{1}{2}}\mathbf{h},\mathbf{e}_{n} \right)=0.
\end{equation*}
This means that $|L|^{-\frac{1}{2}}\mathbf{h}=0$ and taking into account our assumptions, it turns out that $\mathbf{h} = 0$. Therefore "0" cannot be an eigenvalue of the operator $\tilde{Q}(z)+J_{L}, \, \mathrm{Im}z \neq 0$ and, evidently, the same is true for the operator $\left(J_{L}\cdot\tilde{Q}(z)+I\right)$. In other words, "-1" is not an eigenvalue of the operator $J_{L}\cdot\tilde{Q}(z)$. But $J_{L}\cdot\tilde{Q}(z)$ is a compact operator. As follows, $J_{L}\cdot\tilde{Q}(z)+I$ boundedly invertible and so is the operator $\tilde{Q}(z)+J_{L}$, $$\left[\tilde{Q}(z)+J_{L} \right]^{-1}=J_{L}\cdot \left[J_{L}\cdot\tilde{Q}(z)+I \right]^{-1} .$$
We proved that for non-real $z$ the values of operator function $\left[\tilde{Q}(z)+J_{L} \right]^{-1} $ are bounded operators.   

Suppose further that for some non-real $z$ there is a vector $f\in \mathbf{L}_{2}(\mathbf{R}_{3})$ from the linear set $\tilde{G}(z)\mathbf{l}_{2}$ that belongs to the domain $\mathfrak{D}(A)$ of the Laplace operator $A$. This vector as any vector from $\mathfrak{D}(A)$ can be represented in the form $f=R(z)w$ with some $w\in\mathbf{L}_{2}(\mathbf{R}_{3})$ while by our assumption there is a vectors $\mathbf{h}\in \mathbf{l}_{2}$ such that
\begin{equation}\label{lim-a}
	R(z)w-\tilde{G}(z)\mathbf{h}=0.
\end{equation}   

Now recall that for each $w\in\mathbf{L}_{2}(\mathbf{R}_{3})$ and any infinitesimal $\varepsilon>0$ it is possible to find an infinitely smooth compact function $\phi(\mathbf{x})$ which is also equal to zero at some $\eta$-neighborhood of isolated points and all the accumulation points of the set $\left\lbrace \mathbf{x}_{n} \right\rbrace$ on the support of $\phi(\mathbf{x})$ to satisfy the condition   
\begin{equation}\label{lim-b}
	\left| \left(w,\phi \right)_{\mathbf{L}_{2}(\mathbf{R}_{3})} \right|\geq\left(1-\varepsilon \right)\left\| w \right\|^{2}_{\mathbf{L}_{2}(\mathbf{R}_{3}}.
\end{equation}
Taking into account further that for $\phi(\mathbf{r})$, as well as for any smooth compact function,
\begin{equation}\label{laplres}
	\phi(\mathbf{x})= \frac{1}{4\pi}
	\int_{\mathbf{R}_{3}}\frac{e^{i\sqrt{z}|\mathbf{x}-\mathbf{x}^\prime|}}{|\mathbf{x}-\mathbf{x}^\prime|}\left[ -\Delta\phi\left(\mathbf{x}^\prime \right)-z\phi\left(\mathbf{x}^\prime \right)\right]  d\mathbf{x}^\prime, 
\end{equation}
we notice that 
\begin{equation*} \begin{array}{c}
		\left(G(z)w,\left[-\Delta\phi -\bar{z}\phi \right]\right)_{\mathbf{L}_{2}(\mathbf{R}_{3})}=\left(w,\phi\right)_{\mathbf{L}_{2}(\mathbf{R}_{3})}, \\  \left(\tilde{G}(z)\mathbf{h},\left[-\Delta\phi -\bar{z}\phi \right]\right)_{\mathbf{L}_{2}(\mathbf{R}_{3})} =0, \quad \mathbf{h}\in \mathbf{l}_{2}.		
	\end{array}
\end{equation*}
Hence by virtue of (\ref{lim-b}) we conclude that 
\begin{equation}\label{lim-c}
	\begin{array}{c}
		\|R(z)w-\tilde{G}(z)\mathbf{h}\|_{\mathbf{L}_{2}(\mathbf{R}_{3})}\cdot
		\left\| -\Delta\phi-\bar{z}\phi  \right\|_{\mathbf{L}_{2}(\mathbf{R}_{3})} \\ \geq 
		\left|\left( \left[R(z)w-\tilde{G}(z)\mathbf{h}\right],\left[-\Delta\phi-\bar{z}\phi \right] 
		\right)_{\mathbf{L}_{2}(\mathbf{R}_{3})}\right|  =\left|\left(w,\phi\right)_{\mathbf{L}_{2}(\mathbf{R}_{3})}\right| \\ \geq\left(1-\varepsilon\right)
		\left\| w\right\|^{2}_{\mathbf{L}_{2}(\mathbf{R}_{3})}.
	\end{array}
\end{equation}
But in view of (\ref{lim-a})  the last  inequality in (\ref{lim-c}) must necessarily be violated unless  $ w=0 $. Therefore  
\begin{equation}\label{imp6}
\tilde{G}(z)\mathbf{l}_{2} \cap \mathfrak{D}(A)=\{0\}.	
\end{equation}

From (\ref{imp6}) and the established properties of the operator function $R_{L}(z)$ it follows, in particular, that $\ker{R_L(z)}=\{0\}, \, \mathrm{Im}z\neq 0 $. An obvious consequence of (\ref{imp6}) and the established properties of the operator function $R_{L}(z)$ is the fulfillment of condition $\ker{R_L(z)}=\{0\}, \, \mathrm{Im}z\neq 0 $. Indeed by virtue of (\ref{imp6}) the equality $$R(z)w-\tilde{G}(z)\left([ \tilde{Q}(z)+4\pi J_{L}]^{-1}\right)\tilde{G}(\bar{z})^{*}w=0, \; \mathrm{Im}z\neq 0, $$ for some $w\in \mathbf{L}_{2}(\mathbf{R}_{3}$ implies $w=0$ .

The verification of the remaining conditions of  Theorem \ref{kreinnext1} for $R_{L}(z)$ is not difficult and is left to the reader.
\end{proof}
\section{Scatttering matrices}
From now on we will assume that the spectrum of selfadjoint operator $A$ is absolutely continuous  and as long as no mention of another that the operator functions $G(z)$ and $Q(z)$ and  the operator $L$ in the expression (\ref{krein5}) are such that all the conditions of Theorem \ref{kreinnext1} hold and that the difference of the resolvent of selfadjoint perturbation $A_{L}$ of $A$ and the resolvent of $A$ is a trace class operator. Hence for the pair of operators $ A, A_ {L} $, the following assertions of abstract scattering theory are true (for proofs and details, see \cite{ Kat},\cite{Ya}). 

The wave operators $\mathfrak{W}_{\pm}\left(A_{L},A \right)$ defined as strong limits
\begin{equation}\label{wave1}
\mathfrak{W}_{\pm}\left(A_{L},A \right) =s-\underset{t\rightarrow\pm\infty}{\lim}e^{iA_{L}t}e^{-iAt} 
\end{equation}
 exist and are isometric mappings of $\mathcal{H}$ onto the absolutely continuous subspace of $A_{L}$. The wave operators $\mathfrak{W}_{\pm}\left(A_{L},A \right)$ are intertwining for the spectral functions $E_{\lambda}, E^{(L)}_{\lambda}, \, -\infty<\lambda<\infty , $  of the operators $A,A_{L}$ in the sense that $$\mathfrak{W}_{\pm}\left(A_{L},A \right)E_{\lambda}=E^{(L)}_{\lambda}\mathfrak{W}_{\pm}\left(A_{L},A \right).$$
 
   The scattering operator, which is defined as the  product of  wave operators   
 \begin{equation}\label{scat1}
  \mathfrak{S}(A_{L},A)=\mathfrak{W}_{+}\left(A_{L},A \right)^{*}\mathfrak{W}_{-}\left(A_{L},A \right)
   \end{equation} 
  is an isometric operator in $\mathcal{H}$ and
\begin{equation}\label{commut}
	 E_{\lambda}\mathfrak{S}(A_{L},A)=\mathfrak{S}(A_{L},A)E_{\lambda}, \; -\infty<\lambda<\infty .
\end{equation}
 
 Therefore for the spectral representation of $A$ in 
$\mathcal{H}$ as the multiplication operator by $\lambda$ in the direct integral of Hilbert spaces 
$\mathfrak{h}(\lambda)$, $$\mathcal{H}\Rightarrow\int \limits_{-\infty}^{\infty}\oplus\mathfrak{h}(\lambda)d\lambda,$$
the scattering operator $\mathfrak{S}(A_{L},A)$ acts as the multiplication operator by a contractive operator  
function $S(A_{L},A)(\lambda)$, which will be below referred to as the scattering matrix.
 
 These rather general assertions are specified below for close singular perturbations.
 
 Let $\sigma (A)$ denote the spectrum of $A$ and  $\Delta$ is some interval that $\subseteq\sigma(A)$.  We will assume that the part of $A$ on $E(\Delta)\mathcal{H}$ 
 has the Lebesgue spectrum of multiplicity $\mathfrak{n}\leq\infty$.  
 Then there exists an isometric operator $\mathfrak{F}$, which maps $E(\Delta)\mathcal{H}$ onto the space 
 $\mathbb{L}^{2}(\Delta;\mathcal{N})$ of vector function on $\Delta$ with values in the auxiliary Hilbert space $\mathcal{N}, \, \dim\mathcal{N}=\mathfrak{n},$  and such that 
 $\mathfrak{F}A|_{E(\Delta)\mathcal{H}}\mathfrak{F}^{-1}$ is the multiplication operator by independent variable in 
 $\mathbb{L}^{2}(\Delta;\mathcal{N})$. Then  using the notation 
 $$\mathfrak{F}(E(\Delta)f)(\lambda)=\mathbf{f}(\lambda), \quad f\in \mathcal{H}, \; \mathbf{f}(\cdot)\in  
 \mathbb{L}^{2}(\Delta;\mathcal{N}), $$ for any $g$ from the domain of $A$ we can write $$\mathfrak{F}\left(E(\Delta)Af\right)(\lambda)=\lambda\cdot\mathbf{f}(\lambda) $$
 and by (\ref{commut})  for any $f,g\in E(\Delta)\mathcal{H}$ we have
 \begin{equation*} \begin{array}{c}
 	(f,g)_{\mathcal{H}}=\int_{\Delta}\left(\mathbf{f}(\lambda),\mathbf{g}(\lambda) \right) _{\mathcal{N}}d\lambda, \\
 (\mathfrak{S}(A_{L},A)f,g)_{\mathcal{H}}=\int_{\Delta}\left(S(A_{L},A)(\lambda)\mathbf{f}(\lambda),\mathbf{g}(\lambda) \right) _{\mathcal{N}}d\lambda.	
\end{array} \end{equation*}
\begin{thm}\label{suboper}
Let $A$ be a selfadjoint operator in Hilbert space $\mathcal{H}$ with absolutely continuous spectrum, $A_{L}$ be its close selfadjoint perturbation and the resolvent $R_{L}(z)$ of $A_{L}$ admits the representation (\ref{krein5}) with selfadjoint operator $L$ in Hilbert space $\mathcal{K}$,  functions $G(z)$ and $Q(z)$ with values being bounded operators from $\mathcal{K}$ to $\mathcal{H}$ and in $\mathcal{K}$, respectively satisfy all the conditions of Theorem \ref{kreinnext1}.  

Suppose additionally that \begin{itemize} 
\item {for fixed $\gamma>0$, some linearly independent system $\lbrace g_{k}\rbrace_{k\in \mathbb{Z}}\subset \mathcal{H}$ and orthonormal system $\lbrace \mathbf{h}_{j}\rbrace_{j\in \mathbb{Z}}\subset \mathcal{K}$ and a set of numbers $\lbrace b_{jk}\rbrace_{jk\in \mathbb{Z}} $ such that $$\sum_{j,k\in\mathbb{Z}} |b_{jk}|<\infty $$
the operator $G(-i\gamma)$ admits the representation 
\begin{equation}\label{G1}
G(-i\gamma)=\sum_{j,k\in\mathbb{Z}}b_{jk}\left(\cdot\, ,\mathbf{h}_{j} \right)_{\mathcal{K}}g_{k} ;
\end{equation}}

\item{ on some interval $\Delta\subset\sigma(A)$ the operator $A$ has the Lebesgue spectrum of multiplicity $\mathfrak{n}\leq\infty$ and for the operator function $$\Gamma(z)=\left[|L|^{-\frac{1}{2}}Q(z)|L|^{-\frac{1}{2}}+J_{L} \right]^{-1}, \; J_{L}=L\cdot |L|^{-1}, \; \mathrm{Im}z\neq 0, $$
almost everywhere on $\Delta$ there is a weak limit
$$ \Gamma(\lambda +i0)=\underset{\varepsilon\downarrow 0}{\lim}\Gamma(\lambda +i\varepsilon), \; \lambda\in \Delta ,$$ such that
\begin{equation}\label{ess}
\underset{\lambda\in\Delta }{\mathrm{esssup}} \left\|\Gamma(\lambda +i0)\right\|<\infty.	
\end{equation}
 }
\end{itemize}
Then the image $\mathfrak{F}\mathfrak{S}(A_{L},A)\mathfrak{F}^{-1}$ of scattering operator acts in $\mathbb{L}^{2}(\Delta;\mathcal{N})$ as the multiplication by operator function
\begin{equation}\label{suboper1a}\begin{array}{c}
	S(A_{L},A)(\lambda)=I-4\pi i\\ \times \sum_{j,k\in\mathbb{Z}}\left( \Gamma(\lambda+i0)\mathbf{h}_{j}, \mathbf{h}_{k} \right)_{\mathcal{K}}\left(\cdot  ,	\tilde{\mathbf{g}}_{j}(\lambda)\right)_{\mathcal{N}}\tilde{\mathbf{g}}_{k}(\lambda),  \\ \tilde{\mathbf{g}}_{j}(\lambda)=\sqrt{2\pi}(\lambda+i\gamma)^{-1}\sum_{l\in\mathbb{Z}}b_{jl}\mathbf{g}_{l}(\lambda).
	
\end{array}
\end{equation} 				
\end{thm}
\begin{proof}

Notice, that the existence of strong limits in (\ref{wave1}) ensures the validity 
of relations 
\begin{equation}\label{wave2}\begin{array}{c}
		\mathfrak{W}_{\pm}\left(A_{1},A \right) =s-\underset{\varepsilon\downarrow 
			0}{\lim}\:\varepsilon\int\limits_{0}^{\infty}e^{-\varepsilon t}e^{\pm iA_{L}t}e^{\mp iAt} \\ 
		=s-\underset{\varepsilon\downarrow 0}{\lim}\:\varepsilon\int\limits_{0}^{\infty}e^{-\varepsilon t}e^{\pm iA_{L}t}\int 
		\limits_{-\infty}^{\infty}e^{\mp i\lambda t}dE_{\lambda} =s-\underset{\varepsilon\downarrow 0}{\lim}\:\pm i\varepsilon 
		\int \limits_{-\infty}^{\infty}R_{L}(\lambda \pm i\varepsilon)dE_{\lambda}.
	\end{array}
\end{equation}
By (\ref{wave2}) the quadratic form of $\mathfrak{S}(A_{L},A)$ for any $f_{1},f_{2}\, \in \mathcal{H}$ can be written as follows
  \begin{equation}\label{scat2}
  \begin{array}{c}
  \left(\mathfrak{S}(A_{L},A)f_{1},f_{2} \right)=\left(\mathfrak{W}_{-}\left(A_{L},A \right)f_{1},\mathfrak{W}_{+}\left(A_{L},A 
  \right)f_{2} \right) \\ =\underset{\varepsilon ,\eta\downarrow 
  0}{\lim}\:-\eta\varepsilon\iint\limits_{-\infty}^{\infty} \left(R_{L}(\lambda + i\varepsilon)dE_{\lambda}f_{1}, R_{L}(\mu 
  - i\eta)dE_{\mu}f_{2}h\right)  \\ = \underset{\varepsilon ,\eta\downarrow 
  0}{\lim}\:-\eta\varepsilon\iint\limits_{-\infty}^{\infty} \frac{1}{\lambda-\mu+i(\varepsilon-\eta)}\left(\left[ 
  R_{L}(\lambda + i\varepsilon)-R_{L}(\mu + i\eta)\right]dE_{\lambda} f_{1}, dE_{\mu}f_{2}\right).
  \end{array}
  \end{equation}
 To proceed to the limits in (\ref{scat2}) first we note that 
\begin{equation}\label{triv1}\begin{array}{c}
 \mp i\varepsilon\,R(\lambda \pm i\varepsilon)dE_{\lambda}f=f,\quad \varepsilon>0, \;f\in\mathcal{H},\\
 -\eta\varepsilon\iint\limits_{-\infty}^{\infty} \frac{1}{\lambda-\mu+i(\varepsilon-\eta)}dE_{\mu}\left[ R(\lambda + 
 i\varepsilon)-R(\mu + i\eta)\right]dE_{\lambda}f \\
=-\eta\varepsilon\iint\limits_{-\infty}^{\infty}dE_{\mu}R(\mu + i\eta)R(\lambda + 
i\varepsilon)dE_{\lambda}f=f.
\end{array}
\end{equation}
Then taking into account that by (\ref{res3}) for the fixed  $\gamma > 0$ the equality 
\begin{equation}\label{res3a}
G(z)=G(-i\gamma)+(z+i\gamma)R(z)G(-i\gamma)
\end{equation}
is true let us use Krein's formula (\ref{krein5}). Then for any $f_{1},f_{2}$ from the domain $\mathcal{D}_{A} $ we obtain the expression
\begin{equation}\label{scat4a}\begin{array}{c}
\left(\mathfrak{S}(A_{L},A)f_{1},f_{2} \right)=(f_{1},f_{2})-	\underset{\varepsilon ,\eta\downarrow 0}{\lim}\:\iint\limits_{-\infty}^{\infty} \frac{[\mu+i(\gamma+\eta)][\lambda-i(\gamma-\varepsilon)]}{\lambda-\mu+i(\varepsilon-\eta)} \\ \times \left(\left[\frac{i\eta}{\mu-\lambda-i\varepsilon}\Gamma(\lambda+i\varepsilon)-\frac{i\varepsilon}{\lambda-\mu-i\eta}\Gamma(\mu+i\eta) \right]G(-i\gamma)^{*}dE_{\lambda}f_{1},  G(-i\gamma)^{*}dE_{\mu}f_{2}\right) 
\end{array}
\end{equation}
But in accordance with (\ref{G1}) for any $f\in\mathcal{H}$ we have
\begin{equation}\label{G2}
	G(-i\gamma)^{*}f=\sum_{j\in\mathbb{Z}}\left(f,\tilde{g}_{j} \right)_{\mathcal{H}}\mathbf{h}_{j}, \quad \tilde{g}_{j}=\sum_{j\in\mathbb{Z}}b_{jk}g_{k} .
\end{equation}
Substituting (\ref{G2})  into (\ref{scat2}) and assuming that $f_{1},f_{2}\in E(\Delta)\mathcal{H}\cap\mathcal{D}_{A} $ yields
\begin{equation}\label{scat4b}\begin{array}{c}
\left(\mathfrak{S}(A_{L},A)f_{1},f_{2} \right)=\int_{\Delta}\left(\mathbf{f}_{1}(\lambda),\mathbf{f}_{1}(\lambda) \right)_{\mathcal{N}} d\lambda  -	\underset{\varepsilon,\eta\downarrow 0}{\lim}\iint_{\Delta\times\Delta}d\lambda d\mu \\ \times \frac{[\lambda-i(\gamma-\varepsilon)][\mu+i(\gamma-\eta)]}{\lambda-\mu+i(\varepsilon-\eta)}  \sum_{k,j\in \mathbb{Z}}\left[\frac{i\eta}{\mu-\lambda-i\varepsilon}\Gamma_{kj}(\lambda+i\varepsilon)-\frac{i\varepsilon}{\lambda-\mu-i\eta}\Gamma_{kj}(\mu+i\eta) \right] \\ \times \left(\mathbf{f}_{1}(\lambda),\mathbf{g}_{j}(\lambda) \right)_{\mathcal{N}}\cdot  \left(\mathbf{g}_{k}(\mu), \mathbf{f}_{2}(\mu)\right)_{\mathcal{N}}, \quad \Gamma_{kj}(z)=\left( \Gamma(z)\mathbf{h}_{j},\mathbf{h}_{k}\right) .
\end{array}	
\end{equation} 

To pass to the limit in (\ref{scat4b}), recall that in the Hilbert space $\mathbb{L}^{2}(\mathbf{R};\mathcal{M})$ of vector functions $\mathbf{f}(\lambda)$ on $\mathbf{R}$ with values in some Hilbert space $\mathcal{M}$ for any $\varepsilon >0$ operators $$\left(\pi_{\pm}^{(\varepsilon)}\mathbf{f} \right)(\lambda)=\pm \frac{1}{2\pi i}\int_{\mathbf{R}}\frac{1}{\mu-\lambda\mp i\varepsilon}\mathbf{f}(\mu)d\mu  $$ are contractions and strong limits $\pi_{\pm}$ of this operators as $\varepsilon\downarrow 0$  are orthogonal projectors  onto the Hardy subspace $\mathbf{H}^{2}_{+}(\mathcal{M})$ in the upper half-plane and its orthogonal complement $\mathbf{H}^{2}_{-}(\mathcal{M})$, respectively. In particular, for any $\mathbf{f}\in \mathbb{L}^{2}(\mathbf{R};\mathcal{M})$ we have
\begin{equation}\label{gen}\begin{array}{c}
\underset{\eta\downarrow 0}{\lim}\underset{\varepsilon\downarrow 0}{\lim}\frac{1}{2\pi i}\int_{\mathbf{R}}\left[ \frac{1}{\mu-\lambda- i\varepsilon}-\frac{1}{\mu-\lambda- i(\varepsilon-\eta)}\right] \mathbf{f}(\mu)d\mu \\	= \underset{\eta\downarrow 0}{\lim}\left[(\pi_{+}\mathbf{f})(\lambda) +\left(\pi_{-}^{(\eta)}\mathbf{f} \right)(\lambda)\right]=\mathbf{f}(\lambda). 
\end{array}
\end{equation}

Let us assume for a moment that the scalar functions $\left(\mathbf{f}_{s}(\lambda),\tilde{\mathbf{g}}_{j}(\lambda) \right)_{\mathcal{N}}, \, s=1,2,$ in (\ref{scat4b}) satisfy the condition
\begin{equation}\label{help3}
	\int_{\Delta}\sum_{j\in\mathbb{Z}}\left|\left(\mathbf{f}_{s}(\lambda),\tilde{\mathbf{g}}_{j}(\lambda) \right)_{\mathcal{N}} \right|^{2}d\lambda<\infty. 
\end{equation}
  In other words, let us assume that vector functions 
  $$\mathbf{f}_{s}(\lambda)=\sum_{j\in\mathbb{Z}}\left(\mathbf{f}_{s}(\lambda),\tilde{\mathbf{g}}_{j}(\lambda) \right)_{\mathcal{N}}\mathbf{h}_{j}$$
  belong to the space $\mathbb{L}^{2}(\Delta;\mathcal{K})$.	 
  In this case, using relations (\ref{gen}), one can directly carry out the passage to the limit in (\ref{scat4b}) and, as a result, obtain the expression
  \begin{equation}
  \begin{array}{c}
  	\left(\mathfrak{S}(A_{L},A)f_{1},f_{2} \right)=\int_{\Delta}\left(\mathbf{f}_{1}(\lambda),\mathbf{f}_{2}(\lambda) \right)_{\mathcal{N}} d\lambda \\ -	\int_{\Delta} \sum_{k,j\in \mathbb{Z}}2i \left( \Gamma(\lambda+i0)\mathbf{h}_{j},\mathbf{h}_{k}\right) _{\mathcal{K}} \left(\mathbf{f}_{1}(\lambda),\tilde{\mathbf{g}}_{j}(\lambda) \right)_{\mathcal{N}}\cdot  \left(\tilde{\mathbf{g}}_{k}(\lambda), \mathbf{f}_{2}(\lambda)\right)_{\mathcal{N}}d\lambda  .
  \end{array}
\end{equation}
It remains to verify that the set of vector functions from  $\mathbb{L}^{2}(\Delta;\mathcal{N})$ satisfying condition (\ref{help3}) is dense in  $\mathbb{L}^{2}(\Delta;\mathcal{N})$. To this end let us take the scalar function 
$$\psi(\lambda)=\sum_{j\in \mathbb{Z}}\left\|\tilde{\mathbf{g}}_{j}(\lambda) \right\|^{2}_{\mathcal{N}} .  $$  
According to the conditions of the theorem, this function is integrable since
$$\begin{array}{c}
\sum_{j\in \mathbb{Z}}\int_{\Delta}\left\|\tilde{\mathbf{g}}_{j}(\lambda) \right\|^{2}_{\mathcal{N}}d{\lambda} =\sum_{j\in \mathbb{Z}}\int_{\Delta}\sum_{k^{\prime},k\in\mathbb{Z}}\bar{b}_{jk^{\prime}}b_{jk}\left(\mathbf{g}_{k}(\lambda),\mathbf{g}_{k^{\prime}}(\lambda) \right)d\lambda \\ =\sum_{j,k\in\mathbb{Z}}\left|b_{jk} \right|^{2} \leq \left(\sum_{j,k\in\mathbb{Z}}\left|b_{jk} \right| \right)^{2}<\infty. \end{array}
$$
For any $\mathbf{f}\in\mathbb{L}^{2}(\Delta;\mathcal{N})$ and any $\delta>0$ put $$ \mathbf{f}_{\delta}=\frac{1}{\sqrt{1+\delta\psi(\lambda)}}\mathbf{f}(\lambda).$$ Then, on the one hand, $$\int_{\Delta}\sum_{j\in\mathbb{Z}}\left|\left(\mathbf{f}_{\delta}(\lambda),\tilde{\mathbf{g}}_{j}(\lambda) \right)_{\mathcal{N}} \right|^{2}d\lambda\leq \int_{\Delta}\frac{\psi(\lambda)}{1+\delta\psi(\lambda)}\left\|\mathbf{f}(\lambda) \right\|^{2}_{\mathcal{N}}d\lambda\leq\frac{1}{\delta}\left\|\mathbf{f} \right\|^{2}  <\infty ,$$ and on the other, it is obvious that $$\underset{\delta\downarrow 0}{\lim}\left\|\mathbf{f}-\mathbf{f}_{\delta} \right\| =0 $$	
\end{proof}
\begin{rem}[\cite{AdP1}]
By our assumption the scattering operator $\mathfrak{S}(A_{1},A)$ is unitary. Therefore the scattering matrix $S(A_{1},A)(\lambda)$ is to be  unitary almost everywhere on $\Delta$. Notice, that the fact that $$S(A_{1},A)(\lambda)^{*}S(A_{1},A)(\lambda)=I$$ follows directly from the relation
\begin{equation}\label{gramm}\begin{array}{c}
	\frac{1}{2i}\left(  \Gamma(\lambda+i0)^{-1}-\left[\Gamma(\lambda+i0)^{*}\right]^{-1}\right)_{j,k\in\mathbb{Z}} \\ =\frac{1}{2i}\left(  |L|^{-\frac{1}{2}}Q(\lambda+i0)|L|^{-\frac{1}{2}}-|L|^{-\frac{1}{2}}Q(\lambda+i0)^{*}|L|^{-\frac{1}{2}}\right)_{j,k\in\mathbb{Z}} \\ =\pi\left(\left(\tilde{\mathbf{g}}_{j}(\lambda),\tilde{\mathbf{g}}_{k}(\lambda) \right)_{\mathcal{N}} \right)_{j,k\in\mathbb{Z}}.
\end{array} 
\end{equation}
and general 
\begin{prop}\label{local2}
	Let $g_{1},...,g_{N}, \, N<\infty,$ be a set of linearly independent vectors of the Hilbert space $ \mathcal{N}$, $\Upsilon=(\gamma_{nm})_{1}^{N}$ is the corresponding Gram-Schmidt matrix and $\Lambda$ is any Hermitian $N\times N$ - matrix such that the matrix $\Lambda+i\Upsilon$ is invertible. Then the matrix 
	\begin{equation}\label{local1}
		\Omega=I-2i \sum\limits_{n,m=1}^{N}\left([\Lambda+i\Upsilon)]^{-1}\right)_{jk}(.\, ,\, g_{j})g_{k}
	\end{equation}
	is unitary.
\end{prop}
The unitarity of the matrix $\Omega$ is verified by direct calculation. The extension of the statement of Proposition \ref{local2} to the case of infinite sequences of linearly independent vectors $\{g_{j}\}_{j\in\mathbb{Z}} $ from $\mathcal{N}$ under the additional conditions of boundedness and bounded invertibility of the operator in $\mathbf{l}_{2} $ generated by the infinite matrix $\Lambda+i\Upsilon$ is left to the reader.  
\end{rem} 

\section{Scattering matrices for perturbations of the Laplace operator by infinite sums of zero-range potentials}
The standardly defined Laplace operator $A=-\Delta$ in $\mathbf{L}_{2}(\mathbf{R}_{3})$ has uniform Lebesgue spectrum of infinite multiplicity, which coincides with the half-axis  $[0,\infty)$ of the complex plane. 

Let $\mathbf{S}_{2}$ be the unit sphere in $\mathbf{R}_{3}$. The unitary mapping 
\begin{equation}\label{spectr} \begin{array}{c}
(\mathfrak{F}f))(\lambda,\mathbf{n})=\frac{\sqrt[4]{\lambda}}
{\sqrt{2}(2\pi)^{\frac{3}{2}}}\int_{\mathbf{R}_{3}}
f(\mathbf{x})e^{-i\sqrt{\lambda}(\mathbf{n}\cdot\mathbf{x})}d\mathbf{x}, \\ f\in\mathbf{L}_{2}(\mathbf{R}_{3}), \, \mathbf{n}\in\mathbf{S}_{2},
\end{array}	
\end{equation} 
of $\mathbf{L}_{2}(\mathbf{R}_{3})$ onto the space of vector function $\mathbf{L}_{2}\left(\mathbf{R}_{+};\mathbf{L}_{2}(\mathbf{S}_{2}) \right) $ with values from the Hilbert space $\mathbf{L}_{2}(\mathbf{S}_{2})$ transforms the Laplace operator $A$ into the selfadjoint operator of multiplication by the independent variable $\lambda$.

Let us consider a sequence $\{\mathbf{x}_{m}\}_{m\in \mathbb{Z}_{+}}$  of different points of $\mathbf{R}_{3}$ that may have only finite number of accumulation points in compact domains of $\mathbf{R}_{3}$ and  a coresponding sequence $(\delta_{m})_{m\in \mathbb{Z}_{+}}>0$  of positive numbers 
$$
	\delta_{0}=|\mathbf{x}_{0}|; \; \delta_{m}=\underset{0\leq j\neq k \leq m}{\min}|\mathbf{x}_{j}-\mathbf{x}_{k}|,  \; m=1,2,\ldots . 
$$
Let $(w_{m})_{m\in \mathbb{Z}{+}}$ be a sequence of non-zero real numbers such that
$$
\sum_{m\in \mathbb{Z}_{+}}\frac{1}{\sqrt{|w_{m}|}}<\infty, \quad  \sum_{m\in \mathbb{Z}_{+}}\frac{1}{\sqrt{|w_{m}|}\delta_{m}}<\infty.
$$ 
By virtue of Theorem \ref{help1}, substituting the diagonal matrix $L=\left(w_{m}\delta_{mn} \right)_{mn\in\mathbb{Z}_{+}}$ generated by this sequence into Krein's formula (\ref{krein7}) for the Laplace operator yields the resolvent $R_{L}$ of  close singular perturbation $A_{L}$ of the Laplace operator. 

To obtain an explicit expression of the scattering matrix $S(A_{L},A)(\lambda)$ for the pair $A_{L},A$ we note that in this case $\mathbf{L}_{2}(\mathbf{S}_{2})$ and the mentioned above space $\mathfrak{l}_{2}$ may successfully act as the Hilbert space $\mathcal{H}$ and the auxiliary space $\mathcal{K}$, respectively in Krein's formula (\ref{krein7}) and the canonical basis $\mathbf{e}_{j}=(\delta_{jk} )_{j,k\in \mathbb{Z}} $ in $\mathfrak{l}_{2}$ may be engaged there as the orthonormal basis $\{\mathbf{h}_{j}\}$. 

If the set of carriers $\left(\mathbf{x}_{j}\right)_{1\leq j\leq N}$ of the singular perturbation is finite, i.e.  $0\leq N <\infty$, $\left (w_{j}\right)_{1\leq j\leq N}$ are any real numbers and $L_{N}=\left(w_{m}\delta_{mn} \right)_{m,n=1}^{N}$ is , then the absolutely continuous spectrum of the corresponding singular perturbation $A_{L_{N}}$ of $A$ fills the half-axis $(0,\infty)$ and the expression (\ref{suboper1a}) for the scattering matrix $S_{N}(\lambda), \, \lambda>0,$ for the pair $(A_{L_{N}})$  has form \cite{AGHH}
\begin{equation}\label{smN}\begin{array}{c}

S_{N}(\lambda;\mathbf{n},\mathbf{n}^{\prime})=\delta(\mathbf{n},\mathbf{n}^{\prime})-\frac{\sqrt{\lambda}}{8i\pi^{2}}\sum\limits_{ j,j^{\prime}=0}^{N} \Gamma_{j,j^{\prime}}(\lambda +i0) \\ \times e^{i\sqrt{\lambda}\mathbf{n}\cdot{\mathbf{x}_{j}}}\cdot e^{-i\sqrt{\lambda}\mathbf{n^{\prime}}
\cdot \mathbf{x}_{j^{\prime}}}.\end{array}	
\end{equation}
Setting
\begin{equation}
	\mathbf{q}_{j}(\mathbf{n})=\frac{\sqrt[4]{\lambda}}{4\pi}e^{i\sqrt{\lambda}(\mathbf{n}\cdot \mathbf{x}_{j})}
\end{equation}
we see that that the scattering matrix  $S_{N}(\lambda)$ in this case is an operator in $\mathbf{L}_{2}(\mathbf{S}_{2})$ of the form
\begin{equation}\label{smabN}
	S_{N}(\lambda)=I-2i\sum\limits_{ j,j^{\prime}=1}^{N}\left(\left[\mathfrak{Q}_{N}(\lambda)+i\mathfrak{G}_{N}(\lambda) \right]^{-1} \right)_{j,j^\prime}\left(\cdot ,\mathbf{q}_{j}(\cdot)  \right)\mathbf{q}_{j^{\prime}}(\mathbf{n}), 
\end{equation}
where $I$ is the unity operator in $\mathbf{L}_{2}(\mathbf{S}_{2})$, $\mathfrak{Q}_{N}(\lambda)$ is a Hermitian $N\times N $-matrix function and $\mathfrak{G}_{N}(\lambda)$ is the Gramm-Schmidt matrix for the set of vectors\newline $\{\mathbf{q}_{j}(\cdot)\}_{j=1}^{N}\subset \mathbf{L}_{2}(\mathbf{S}_{2})$, i.e.  $$\mathfrak{G}_{N}(\lambda)=\left(\left(\mathbf{q}_{j},\mathbf{q}_{j^{\prime}} \right)_{\mathbf{L}_{2}(\mathbf{S}_{2})} \right)_{j,j^{\prime}=1}^{N}. $$ 	

Note that according to (\ref{smabN})
\begin{equation}\label{smfcN}
	S_{N}(\lambda)\mathbf{q}_{s}=\sum\limits_{ j=0}^{N}\left(\left[\mathfrak{Q}_{N}(\lambda)-i\mathfrak{G}_{N}(\lambda) \right]\cdot \left[\mathfrak{Q}_{N}(\lambda)+i\mathfrak{G}_{N}(\lambda) \right]^{-1} \right)_{s,j}\mathbf{q}_{j}.
\end{equation}

Let us potentiate the perturbation $A_{L_{N}}$ by adding at extra point $\mathbf{x}_{N+1}$ one more point potential with some real parameter $w_{N+1}$. To compare the arising scattering matrix $S_{N+1}(\lambda)$ with the preceded $S_{N}(\lambda)$ remember the following proposition 

\textit{Let $\mathfrak{W}_{N}$ and $\mathfrak{W}_{N+1}$ be invertible $N\times N$ and $(N+1)\times(N+1)$  matrices respectively and the left upper $N\times N$ block of  $\mathfrak{W}_{N+1}$ coincides with $\mathfrak{W}_{N}$. Let $\{\mathfrak{W}\}_{N}^{-1}$ be the $(N+1)\times(N+1)$ matrix the upper left $N\times N$ block of which is the matrix $\mathfrak{W}_{N}^{-1}$  and other its entries are zeros and $\mathbf{e}_{N+1}$ be the $(N+1)\times 1$-matrix (column-vector) all the entries of which are zeros excepth the lowest one that equals $1$. Then} 
\begin{equation}\label{auxnew1}
\mathfrak{W}_{N}^{-1}=\mathfrak{W}_{N+1}^{-1}-\left[\mathbf{e}_{N+1}^{T} \mathfrak{W}_{N+1}^{-1}\mathbf{e}_{N+1}\right]^{-1}\mathfrak{W}_{N+1}^{-1}\mathbf{e}_{N+1}\mathbf{e}_{N+1}^{T}\mathfrak{W}_{N+1}^{-1} 
\end{equation}	
As follows in our case for the corresponding matrices in (\ref{smabN}) we have
\begin{equation}\begin{array}{c}
\left[\mathfrak{Q}_{N+1}(\lambda)+i\mathfrak{G}_{N+1}(\lambda)\right]^{-1}=\left[{\mathfrak{Q}}_{N}(\lambda)+
i\mathfrak{G}_{N}(\lambda) \right]^{-1} \\
+\left[\mathbf{e}_{N+1}^{T}\left[\mathfrak{Q}_{N+1}(\lambda)+i\mathfrak{G}_{N+1}(\lambda)\right]^{-1}\mathbf{e}_{N+1} \right]^{-1} \\ \times \left[\mathfrak{Q}_{N+1}(\lambda)+i\mathfrak{G}_{N+1}(\lambda)\right]^{-1}\mathbf{e}_{N+1}\mathbf{e}_{N+1}^{T}\left[\mathfrak{Q}_{N+1}(\lambda)+i\mathfrak{G}_{N+1}(\lambda)\right]^{-1}.
\end{array} 
\end{equation}
and setting further
\begin{equation}\begin{array}{c}
\left(\begin{array}{c} \xi_{1} \\ \dots \\ \xi_{N+1}  
\end{array}\right)=\left[\mathfrak{Q}_{N+1}(\lambda)+i\mathfrak{G}_{N+1}(\lambda)\right]^{-1}\mathbf{e}_{N+1}, \\
\mathfrak{f}_{N+1}(\mathbf{n})= \sum\limits_{j=1}^{N+1}\xi_{j}\mathbf{q}_{j}(\mathbf{n}),\\ \mathfrak{d}_{N+1}(\lambda)=\frac{\det\left[\mathfrak{Q}_{N+1}(\lambda)+i\mathfrak{G}_{N+1}(\lambda)\right]}{\det\left[\mathfrak{Q}_{N}(\lambda)+i\mathfrak{G}_{N}(\lambda)\right]}
\end{array}\label{not8}
\end{equation}
we conclude that
\begin{equation}\label{recs}
S_{N+1}(\lambda ;\mathbf{n},\mathbf{n}^{\prime})=S_{N}(\lambda;\mathbf{n},\mathbf{n}^{\prime}) -2i\mathfrak{d}_{N+1}(\lambda) \mathfrak{f}_{N+1}(\mathbf{n})\overline{\mathfrak{f}_{N+1}(\mathbf{n}^{\prime})}. 
\end{equation}

Note that by (\ref{not8}) 
\begin{equation}\label{not9}\begin{array}{c}
\xi_{j}=-\mathfrak{d}_{N+1}(\lambda)^{-1}\sum\limits_{k=1}^{N}
\left(\left[\mathfrak{Q}_{N}(\lambda)+i\mathfrak{G}_{N}(\lambda)\right]^{-1}\right)_{jk}g(z;\mathbf{x}_{k}-\mathbf{x}_{N+1}), \, j=1,...,N; \\ \xi_{N+1}=\mathfrak{d}_{N+1}(\lambda)^{-1}.	
\end{array}
\end{equation}
Note also that by virtue of (\ref{not8}), (\ref{not9})
\begin{equation}\begin{array}{c}\label{not10}
\sum\limits_{k=1}^{N}\xi_{j}\left(S_{N}^{*}(\lambda)\mathbf{q}_{j} \right)(\mathbf{n})=-\mathfrak{d}_{N+1}(\lambda)^{-1} \\ \times\sum\limits_{j,k=1}^{N}
\left(\left[\mathfrak{Q}_{N}(\lambda)-i\mathfrak{G}_{N}(\lambda)\right]^{-1}\right)_{jk}g(z;\mathbf{x}_{k}-\mathbf{x}_{N+1})
\mathbf{q}_{j}(\mathbf{n})
\end{array} 
\end{equation}
and by (\ref{smabN})
\begin{equation}\label{not11}\begin{array}{c}
S_{N}^{*}(\lambda)\mathbf{q}_{N+1}=\mathbf{q}_{N+1} \\ +2i\sum\limits_{ j,k=1}^{N}\left(\left[\mathfrak{Q}_{N}(\lambda)-i\mathfrak{G}_{N}(\lambda) \right]^{-1} \right)_{j,k}\left[ \mathrm{Im}g(z;\mathbf{x}_{k}-\mathbf{x}_{N+1})\right] \mathbf{q}_{j}(\mathbf{n}).	
\end{array}
\end{equation}	
It follows from (\ref{not9})-(\ref{not11}) that

\begin{equation}\label{not12 }\begin{array}{c}
\left( S_{N}^{*}(\lambda)\mathfrak{f}_{N+1}\right)(\mathbf{n})=
\mathfrak{d}_{N+1}(\lambda)^{-1}\mathbf{q}_{N+1}(\mathbf{n}) \\-\mathfrak{d}_{N+1}(\lambda)^{-1}\sum\limits_{j,k=1}^{N}
\left(\left[\mathfrak{Q}_{N}(\lambda)-i\mathfrak{G}_{N}(\lambda)\right]^{-1}\right)_{jk}\overline{g(z;\mathbf{x}_{k}-\mathbf{x}_{N+1})}
\mathbf{q}_{j}(\mathbf{n}).
\end{array}
\end{equation}

and by virtue of above relations we conclude that
\begin{equation}\begin{array}{c}
S_{N+1}(\lambda)=S_{N}(\lambda)\left[I-2i\mathfrak{d}_{N+1}(\lambda)^{-1}\left(\cdot,\mathbf{w}_{\lambda , N+1}\right)\mathcal{J}\mathbf{w}_{\lambda, N+1}\right], \\  \mathbf{w}_{\lambda , N+1}(\mathbf{n})=\mathbf{q}_{N+1}(\mathbf{n}) -\sum\limits_{j,k=1}^{N}
\left(\left[\mathfrak{Q}_{N}(\lambda)-i\mathfrak{G}_{N}(\lambda)\right]^{-1}\right)_{jk}\overline{g(z;\mathbf{x}_{k}-\mathbf{x}_{N+1})}
\mathbf{q}_{j}(\mathbf{n}), \\ \mathcal{J}\mathbf{w}(\mathbf{n})=\overline{\mathbf{w}(\mathbf{-n})}	
\end{array}
\end{equation}  
and hence
\begin{equation}\label{last}
\det S_{N+1}(\lambda)=\det S_{N}(\lambda)\left[1-2i\mathfrak{d}_{N+1}(\lambda)\left(\mathcal{J}\mathbf{w}_{\lambda, N+1}, \mathbf{w}_{\lambda, N+1}\right) \right].  
\end{equation}


\begin{thebibliography}{1}
\bibitem{Ad} V. Adamyan, \textit {Singular perturbations of unbounded selfadjoint operators. Reverse approach}, Operator Theory: Advances and Applications, \textbf{276} ( 2020), 63 - 79
\bibitem{AdP} V. Adamyan and B. Pavlov, \textit{Null--range Potentials and M.G. Krein's formula for generalized resolvents}, Zap. Nauchn. Sem. Leningrad. Otd. Matemat. Inst. im. V.F. Steklova 
\textbf{149} (1986), 7--23 (J. of Soviet Math. \textbf{42}(1988), 1537--1550).
\bibitem{AdP1} V. Adamyan and B. Pavlov, \textit{Local Scattering Problem and a Solvable Model of Quantum Network}, Operator Theory: Advances and Applications, \textbf{198} ( 2009), 1-10
\bibitem{AGHH}Albeverio, Sergio; Gesztesy, Friedrich; Høegh-Krohn, Raphael; Holden, Helge: Solvable models in quantum mechanics. Texts and Monographs in Physics. 2nd edition, AMS-Chelsea Series, Amer. Math. Soc., 2005
\bibitem{AK} S. Albeverio and P. Kurasov, \textit{Singular Perturbations of Differential Operators},
Cambridge University Press, 2000.
\bibitem{BF} F.A. Berezin and L.D. Faddeev, \textit{Remarks on the Schr\"{o}dinger equation with singular potential}, Doklady Akad. Nauk SSSR, \textbf{137}, (1961), 1011 - 1014 
\bibitem{BMN} J. Behrndt, M. Malamud and H. Neidhardt, \textit{Scattering matrices and Dirichlet-to-Neumann maps}, Journal of Funct. Anal., \textbf{ 273}, (2017), 1970-2025
\bibitem{DeOst} Yu.N. Demkov and V.N. Ostrovsky, \textit{Zero-range Potentials and their Applications in Atomic Physics}, Plenum, New York, 1988  
\bibitem{GP} N.I. Gerasimenko and B.S. Pavlov, \textit{Scattering problems on noncompact graphs},  Theoret. and Math. Phys., \textbf{74}:3 (1988), 230–240 
\bibitem{GoKr} I. Gohberg and M. Krein, \textit{Introduction to the Theory of Linear Nonselfadjoint Operators}, Translations of mathematical monographs, v.\textbf{18}, American Mathematical Soc., 1988
\bibitem{GHM} A. Grossmann, R. H\o{}egh-Krohn and M. Mebkhout, \textit{A class of explicitely soluble, local, many-center Hamiltonians  for one-particle quantum mechanics in two and three dimensions}. I. J. Math. Physics, textbf{21}(9) (1980), 2376-2385
\bibitem{H} P. Halmos, \textit{A Hilbert Space Problm Book}, D. van Nostrand Company, Inc., Princeton/New Jersey, Toronto, London (1967)  
\bibitem{Kat}T. Kato,\textit{Perturbation Theory for Linear Operators}, Springer-Verlag, Berlin, second edition, 1976 
\bibitem{KM} A.S. Kostenko and M.M. Malamud, \textit{1-D Schr\"{o}dinger operator with point local interactions on a discrete set}, J. Differential Equations, textbf{249} (2010) 253–304
\bibitem{Kr} M.G. Krein, \textit{Concerning the resolvents of an Hermitian operator with the deficiency-index $(m,m)$}, Doklady Acad. Sci. USSR, \textbf{52}, (1946), 651-654 
\bibitem{Pav} B.S. Pavlov, \textit{The theory of extensions and explicitly solvable models}, Uspekhi Mat. Nauk, \textbf{42}, 6(258)(1987), 99 - 131
\bibitem{PF}B.S. Pavlov and M.D. Faddeev,  \textit{Model of free electrons and the scattering problem}, Theoret. and Math. Physics, \textbf{55}:2(1983), 485 - 492
 \bibitem{Ya} D.R. Yafaev, \textit{Mathematical Scattering Theory}, Translations of Mathematical Monographs, American Math. Soc., Providence, RI, \textbf{105} (1992)     
\end{thebibliography}
\end{document}